\documentclass[11pt, a4paper]{article}

\usepackage[margin=1in]{geometry}
\usepackage[utf8]{inputenc}
\usepackage[english]{babel}
\usepackage{graphicx}
\usepackage{amsmath, amssymb, amsthm} % Standard math/theorem packages
\usepackage{booktabs}
\usepackage{multirow}
\usepackage{longtable}
\usepackage{float}
\usepackage{url}
\usepackage{xcolor}
\usepackage{siunitx}
\usepackage{stfloats}

\usepackage{setspace}
\usepackage{tikz}
\usetikzlibrary{positioning, arrows.meta, patterns, calc, shapes.geometric}

\usepackage[authoryear,sort]{natbib}
\newtheorem{theorem}{Theorem}

\newtheorem{proposition}{Proposition}
\newtheorem{insight}{Insight}

\newtheorem{corollary}[theorem]{Corollary} 

\allowdisplaybreaks
\numberwithin{equation}{section} % Replaces \EquationsNumberedThrough

\title{\textbf{On the Entropy in Last-Mile Logistics}}
\author{
    Berry Gerrits \\ 
    \small University of Twente \\ 
    \small \texttt{b.gerrits@utwente.nl}
    \and
    Wouter van Heeswijk \\ 
    \small University of Twente \\ 
    \small \texttt{w.j.a.vanheeswijk@utwente.nl}
}
\date{February 2026}

\begin{document}

\maketitle

\begin{abstract}
Last-mile logistics (LML) is characterized by high fragmentation, yet existing research treats this as an exogenous constraint rather than a quantifiable and optimizable system property. This paper introduces a framework for measuring LML complexity using structural entropy, derived from Boltzmann's statistical mechanics. Unlike traditional KPIs such as distance or cost, structural entropy quantifies the cardinality of the configuration space, providing a diagnostic of inherent system disorder. We establish a formal duality with Shannon entropy, linking absolute complexity burden to distributional balance.

We apply our entropy framework to 6,112 Amazon last-mile routes across five U.S. cities. Current operations exhibit persistently high normalized entropy, indicating near-maximal fragmentation with limited room for algorithmic routing improvement. A stable non-linear scaling relationship between entropy and route distance validates the metric as a predictive indicator of operational difficulty. To evaluate spatial consolidation, we develop a system-wide entropy measure accounting for all movements by both carriers and customers. We establish a theoretical conservation principle: under idealized conditions, spatial consolidation merely redistributes entropy from carrier to customer. Both idealizing conditions are systematically violated in practice, thereby actually increasing total system entropy (including both carrier and customer perspectives). Our system-wide measure reveals that from the carrier perspective, spatial consolidation reduces carrier entropy by up to 40\% under aggressive adoption, but increases total system entropy by activating customer collection trips, though trip chaining can diminish or reverse this effect. Temporal consolidation, by contrast, genuinely reduces entropy by decreasing delivery events without creating new movements. By formalizing fragmentation as a measurable structural property, this research provides a new lens for network design, consolidation policy, and evaluation last-mile system performance.
\end{abstract}


\section{Introduction}
\label{sec:intro}

Last-mile logistics (LML) is the most economically volatile segment of the supply chain, with expenses accounting for up to 53\% of total costs \citep{lim2025cutting} due to its inherently fragmented and dispersed nature \citep{savelsbergh201650th, vanheeswijk2020evaluating}. Despite this, fragmentation is almost universally treated as an exogenous system property \citep{su14020911}. Prevailing optimization models focus on minimizing operational realizations---primarily cost, distance, or environmental impact \citep{vanheeswijk2020evaluating, halldorsson2020last, brochado2024performance}---yet they fail to quantify the underlying structural complexity that governs these outcomes. Two last-mile systems may yield identical distance and service levels, yet possess fundamentally different configurations of demand that respond divergently to exogenous factors such as demand fluctuation, policy shifts and operational disruptions. We argue that fragmentation is a fundamental system property that requires formal measurement and, ultimately, strategic optimization.

Existing industry solutions, such as urban consolidation centers, pickup points, and parcel lockers, intuitively target fragmentation by concentrating demand in space \citep{Mangiaracina2019,vanheeswijk2019delivery}. While these interventions are celebrated for reducing carrier stop counts \citep{janjevic2017investigating} and mitigating not-at-home risks, they are often evaluated through the lenses of economies of scale, emissions reduction, or traffic mitigation \citep{wang2018innovation}. What remains absent is a metric that captures the structural transformation of the network itself.

This paper addresses this gap by introducing \textit{structural entropy} as a novel indicator for last-mile systems. Drawing from the statistical mechanics of \cite{Clausius1879Mechanical}, we posit that entropy provides a mathematically rigorous framework to quantify the degree of disorder within a delivery network. By mapping delivery demand onto discrete spatial zones and temporal slots, entropy complements traditional KPIs (cost, distance, lead time) by revealing inefficiencies obscured by aggregation. Unlike Shannon entropy \citep{shannon1948} used to quantify street orientations \citep{gudmundsson2013entropy} or Maximum Entropy Models used in demand estimation \citep{abbas2006maximum}, and distinct from entropy-regulation in search heuristics \citep{ahmed-entropy-regulation}, our Boltzmann-based approach \citep{boltzmannentropy} quantifies the \textit{cardinality of the configuration space}---the number of distinct ways a logistics system can be arranged. Moreover, widely used KPIs---cost, distance, and service level \citep{VegaMejia2019}, and stop count---measure operational \textit{outcomes} but leave the \textit{structure} of the delivery problem unmeasured. This distinction is critical: 

\begin{itemize} 
\item \textbf{Distance} reflects routing efficiency but is confounded by the performance of the chosen algorithm; entropy measures the structural dispersion of demand independent of the route. 
\item \textbf{Cost} aggregates structural inefficiencies with controllable decisions like vehicle selection or pricing; entropy provides a diagnostic of the underlying demand configuration.
\item \textbf{Service level} monitors reliability but fails to explain \textit{why} a system is difficult to maintain; entropy exposes the structural constraints that limit operational flexibility. 
\item \textbf{Stop count} records the number of distinct delivery location but is blind to how parcels are distributed across them and subsequently the fragmentation that exists, i.e., whether parcels cluster at a few stops or spread uniformly across all of them.
\end{itemize}

By isolating the structural component, entropy reveals vulnerabilities---such as susceptibility to demand fluctuations---that traditional metrics mask. Furthermore, entropy is distinct from network topology metrics like node connectivity or density \citep{anderson, amaral2020}, as it does not account for the actual flow and volume of demand (e.g., full-truck load versus less-than-full). It also transcends operational fairness metrics, such as Gini coefficients or variation in route duration \citep{matl2017}. While typical logistics optimization (e.g., facility location problem) treats dispersion as a fixed input to be covered \citep{snyder2004}, we treat it as a property to be quantified and managed.

The contribution of this work is threefold. First, we define entropy in the context of last-mile logistics, demonstrating how it reflects fragmentation across various demand states. Second, we formally analyze system behavior under diverse scenarios, identifying the conditions under which spatial and temporal consolidation reduces systemic disorder. Third, we demonstrate managerial relevance through an empirical study of Amazon's last-mile operations, revealing the inherent trade-offs between efficiency, service, and sustainability.

The remainder of this paper is structured as follows. 
Section \ref{sec:entropy} defines the structural entropy-based KPI, analyses theoretical properties such as the impact of spatial and temporal consolidation, and introduces extensions that handle real-world complexities such as preferred parcel points, failed deliveries and trip chaining.
Section \ref{sec:case_study} reports computational experiments and applies the proposed framework to Amazon's last-mile operations. Finally, Section \ref{sec:discussion} discusses managerial implications for logistics service providers, policy makers, and consumers, and Section \ref{sec:conclusion} concludes the paper.

\section{Defining Entropy in Last-Mile Logistics}
\label{sec:entropy}
Last-mile logistics (LML) systems are characterized by a high cardinality of potential operational states. We propose that the fundamental challenge of the last mile---fragmentation---can be rigorously quantified by measuring the disorder induced by these states. We adopt the Boltzmann definition of entropy to measure the number of ways a system's states can be arranged, or its \textit{structural complexity}. In this conceptualization, a system with high entropy possesses a large configuration space (high fragmentation), while a low-entropy system is highly constrained and consolidated. 

The remainder of this section is organized as follows. Section~\ref{ssec:structural_entropy} formalizes the definition of structural entropy $G$. Section~\ref{ssec:structural_vs_shannon} establishes the duality between structural complexity and distributional balance, situating structural (Boltzmann) entropy $G$ alongside Shannon entropy. Section~\ref{ssec:consolidation_dynamics} analyzes the fundamental mechanics of consolidation, establishing a theoretical conservation principle, demonstrating that its assumptions are violated in practice, and developing a system-wide entropy measure that accounts for all movements by both carrier and customers. Section \ref{subsec:generalized_entropy} generalizes the definition of structural entropy to accommodate real-life mechanisms, including failed deliveries, pickup point restrictions, direct-to-pickup flows, and customer trip chaining.

\subsection{Structural Entropy as a Measure of Fragmentation}
\label{ssec:structural_entropy}

Let a logistics macrostate be defined by a set of $N$ parcels to be distributed among $K$ discrete service locations (e.g., households, lockers, or hubs). The macrostate is characterized solely by aggregate parcel counts and therefore treats parcels as indistinguishable at the macroscopic level. Let $p_k$ denote the number of parcels assigned to location $k \in \{1, \dots, K\}$, such that
\[
\sum_{k=1}^{K} p_k = N.
\]

To quantify the size of the underlying configuration space, we distinguish between macrostates and microstates. A \emph{microstate} represents a specific assignment of individual (distinguishable) parcels to service locations, whereas the \emph{macrostate} records only the resulting occupancy numbers $\{p_k\}_{k=1}^K$ and ignores parcel identity. The number of distinct microstates $W$ compatible with a given macrostate is therefore given by the multinomial coefficient

\begin{equation*}
W = \frac{N!}{\prod_{k=1}^{K} p_k!}.
\end{equation*}

We define the structural entropy $G$ as the natural logarithm of the configuration cardinality,

\begin{equation*}
G = \ln(W) = \ln(N!) - \sum_{k=1}^{K} \ln(p_k!).
\end{equation*}

Unlike traditional distance-based metrics, $G$ is a structural diagnostic. It quantifies the inherent disorder of the demand distribution before a single vehicle route is constructed.

\subsection{Duality of Fragmentation and Balance}
\label{ssec:structural_vs_shannon}
To fully characterize a logistics network, one must distinguish between the \textit{complexity of the structure} ($G$) and the \textit{balance of the demand distribution}. We address this by situating our structural (Boltzmann) entropy $G$ alongside the classical Shannon entropy $H$. While $G$ counts the number of discrete arrangements, Shannon entropy $H$ quantifies the unpredictability or evenness of the parcel distribution across locations:

\begin{equation*}
H = -\sum_{k=1}^{K} \phi_k \ln(\phi_k), \quad \text{where } \phi_k = \frac{p_k}{N}
\end{equation*}

In LML, an efficient system minimizes $G$ (reducing structural fragmentation) while maintaining an appropriate $H$ (ensuring balanced utilization of infrastructure). The relationship between these two metrics defines a classification of entropy quadrants, as outlined in Table~\ref{tab:entropy_quadrants}.

\begin{table}[ht]
\centering
\caption{Classification of entropy quadrants. In the ideal case, structural (Boltzmann) entropy $G$ is low and Shannon entropy $H$ is high, implying that stops are highly grouped and infrastructure is fully utilized.}
\label{tab:entropy_quadrants}
\begin{tabular}{@{}ll@{}}
\toprule
\textbf{Entropy Profile} & \textbf{System Implication} \\ 
\midrule
Low $G$, High $H$  & Highly grouped stops; infrastructure is fully utilized. \\
High $G$, High $H$ & Extreme dispersion; demand is balanced but fragmented. \\
Low $G$, Low $H$   & Grouped stops, but demand is skewed to a few nodes. \\
High $G$, Low $H$  & Dispersed demand with high local volatility. \\
\bottomrule
\end{tabular}
\end{table}

\begin{proposition}[Scale-Dependency and Asymptotic Coupling]
\label{prop:scaling}
Structural entropy $G$ is an extensive property sensitive to system volume $N$, whereas Shannon entropy $H$ is an intensive, scale-invariant property. For a fixed distribution $\phi_k = p_k/N$, the relationship is governed by:

\begin{equation*}
G \approx N \cdot H
\end{equation*}

\end{proposition}

\begin{proof}
Recall the definition $G = \ln(N!) - \sum_{k=1}^K \ln(p_k!)$. For large $N$, we apply Stirling's approximation, $\ln(n!) \approx n \ln n - n$ (see Appendix~\ref{appendix_stirling} for the full derivation).

\begin{align*} 
G &\approx (N \ln N - N) - \sum_{k=1}^K (p_k \ln p_k - p_k) \\ 
&= N \ln N - N - \sum_{k=1}^K p_k \ln p_k + \sum_{k=1}^K p_k \\ 
&= N \ln N - \sum_{k=1}^K p_k \ln p_k \quad (\text{since } \sum p_k = N) 
\end{align*}

Substituting $p_k = N\phi_k$:

\begin{align*}
G &\approx N \ln N - \sum_{k=1}^{K} (N\phi_k) \ln(N\phi_k)\\
&= N \ln N - \sum_{k=1}^{K} N\phi_k (\ln N + \ln \phi_k)\\
&= N \ln N - N \ln N \sum_{k=1}^{K} \phi_k - N \sum_{k=1}^{K} \phi_k \ln \phi_k \\
&= -N \sum_{k=1}^{K} \phi_k \ln \phi_k \quad (\text{since } \sum \phi_k = 1) \\
&= N \cdot H 
\end{align*}

Thus, $G \approx N \cdot H$. This proves that while $H$ remains constant regardless of $N$ (given a fixed relative distribution), $G$ scales linearly with system volume, capturing the increasing structural complexity of larger networks.
\end{proof}

\begin{insight}[Volume vs. Density]Proposition \ref{prop:scaling} clarifies the diagnostic roles of the two metrics. $H$ identifies the \textit{density} of fragmentation (e.g., how balanced the deliveries are across locations), whereas $G$ identifies the \textit{absolute complexity burden} placed on the carrier. A city that doubles its delivery volume $N$ while maintaining the same relative distribution $\phi$ will see its Shannon entropy $H$ remain unchanged, but its structural entropy $G$ will double, reflecting the intensified routing and sequencing challenge.
\end{insight}

\subsection{Consolidation Dynamics: From Conservation to System-Wide Entropy} \label{ssec:consolidation_dynamics}
Logistics interventions may reduce structural entropy through consolidation along two dimensions: spatial (grouping parcels at shared service points) and temporal (batching deliveries across time windows). To evaluate the net impact of such interventions, we distinguish between two observer perspectives. The carrier observes entropy at the resolution of delivery stops---a pickup point holding $p_k$ parcels registers as a single stop, regardless of the number of customers it serves. The system-wide observer, by contrast, accounts for the full chain from dispatch to final receipt, including the collection activities that spatial consolidation transfers to customers. Whether an intervention achieves a genuine reduction in total structured entropy $G$, or merely redistributes it across the firm boundary, depends on the assignment structure at the interface between carrier and customer---a distinction we develop in three stages below.

\subsubsection{A Theoretical Conservation Principle}
\label{sssec:conservation}
We begin by establishing a theoretical baseline. Under idealized assumptions, spatial consolidation (see Figure~\ref{fig:spatial_consolidation}) does not reduce the total configuration space of the system but instead redistributes entropy across organizational boundaries. This conservation result, while not directly applicable to most real logistics operations, clarifies the mechanism by which spatial consolidation transfers complexity and motivates the system-wide entropy measure that follows.

Consider a theoretical scenario in which \textit{all} $N$ parcels are routed to $K$ shared service points---no home deliveries remain. We impose two symmetry conditions (that we relax in Section~\ref{sssec:breaking_symmetry}): (i) customers are unconstrained in their choice of service point, and (ii) parcels at each service point are indistinguishable from the collection perspective, so that any of the $p_k$ parcels can be matched to any arriving customer.

Under these conditions, the delivery-side entropy $G^{\mathrm{delivery}} = \ln(N!) - \sum_{k=1}^K \ln(p_k!)$ captures the number of ways the carrier can arrange parcels across stops. The collection-side entropy $G^{\mathrm{pickup}} = \sum_{k=1}^K \ln(p_k!)$ captures the $p_k!$ ways parcels can be matched to customers at each service point.

\begin{proposition}[Conservation of Structural Entropy]
\label{prop:conservation}
Under the symmetric conditions stated above, the sum $G^{\mathrm{delivery}} + G^{\mathrm{customer}} = \ln(N!)$ is invariant to the degree of spatial partitioning $K$.
\end{proposition}

\begin{proof}
\begin{align*}
G^{\mathrm{delivery}} + G^{\mathrm{customer}}
&= \left[ \ln(N!) - \sum_{k=1}^K \ln(p_k!) \right]
+ \left[ \sum_{k=1}^K \ln(p_k!) \right]
= \ln(N!).
\end{align*}
\end{proof}

This conservation result reveals that under symmetric conditions, spatial consolidation is a zero-sum transformation: by moving parcels from $N$ doorsteps to $K$ lockers, the carrier exports configuration complexity to the customer base. The service point acts as an interface where delivery-side entropy converts into collection-side entropy.

\begin{figure}[H]
    \centering
    \includegraphics[width=0.5\linewidth]{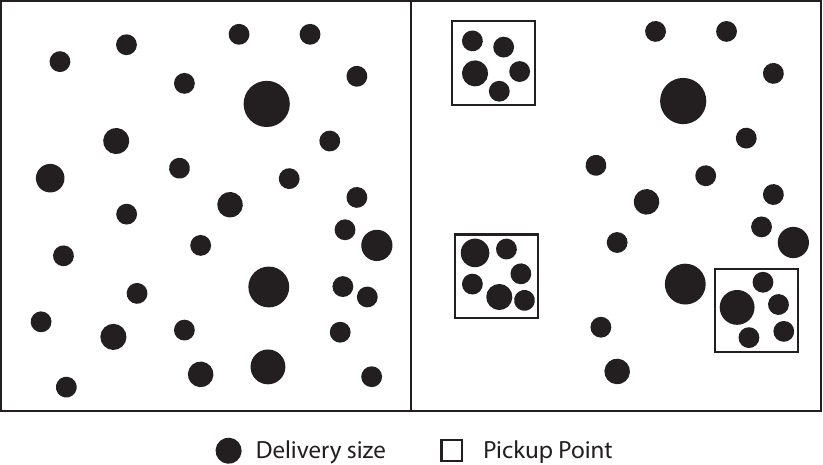}
    \caption{Spatial consolidation through pickup points. Under idealized symmetric conditions, reducing carrier entropy transfers equivalent complexity to the customer-side collection process, conserving total structural entropy.}
    \label{fig:spatial_consolidation}
\end{figure}

\subsubsection{Breaking the Conservation: From Theory to Practice}
\label{sssec:breaking_symmetry}

The conservation result, while not applicable to most last-mile logistics systems, may hold in settings where items at service points are genuinely interchangeable and customers are unconstrained in their collection choice---for instance, prize or promotional item distribution at collection points, vaccine allocation across walk-in clinics, or food bank distribution where identical packages are dispensed to any arriving citizen. Most other last-mile logistics systems violate the two symmetry conditions:

\begin{enumerate}
    \item \textbf{Unconstrained customer choice.} In practice, customers select a distinct service point. This deterministic, constrained assignment eliminates the combinatorial freedom in ``which customer goes where.''
    
    \item \textbf{Indistinguishable parcels.} In practice, every parcel is labeled for a specific customer. When customer $i$ arrives at locker A, they retrieve their own parcel(s)---not any of the $p_k$ parcels stored there. The collection process is fully determined, leaving no arrangement ambiguity at the service point.
\end{enumerate}

When both conditions are broken, the service point no longer functions as a combinatorial interface. There is no distribution problem to solve: each customer is assigned to a specific point and retrieves a specific parcel (or parcels). Customer-side structural entropy vanishes entirely. The delivery-side entropy decreases through consolidation, and nothing compensates on the customer side. Total structural entropy genuinely decreases.

However, this accounting omits an important element. Under home delivery, the logistics system is one-sided: only the carrier traverses the network, and customers remain stationary. When parcels are redirected to service points, a new class of movements emerges: individual customer collection trips. Each customer assigned to a pickup point must travel to the service point and return---a movement that did not exist under home delivery. Structural entropy, which counts arrangements of parcels across locations, is blind to these newly created movements. To capture the full system-wide impact of spatial consolidation, we require a measure that accounts for all movements by all actors, as detailed in Section~\ref{sssec:system_entropy}. 

\subsubsection{System-Wide Entropy}
\label{sssec:system_entropy}

We define the system-wide entropy as the sum of delivery-side and collection-side entropies, where both components count distinct movement arrangements using the same multinomial structure.

Consider $N$ parcels belonging to $C^{\mathrm{total}}$ unique customers. Let $C^{\mathrm{pickup}}$ denote the number of unique customers assigned to pickup points and $C^{\mathrm{home}}$ the number receiving home delivery, with $C^{\mathrm{total}} = C^{\mathrm{pickup}} + C^{\mathrm{home}}$. Note that $C^{\mathrm{pickup}} \leq N_{\mathrm{pickup}}$, since a single customer may have multiple parcels consolidated at one service point.

On the delivery side, the carrier visits $K$ pickup points (with $p_k$ parcels at point $k$) and $C^{\text{home}}$ individual homes (each receiving $q_i$ parcels):

\begin{equation}
\label{eq:delivery_entropy}
    G^{\mathrm{delivery}} = \ln(N!) - \sum_{k=1}^{K} \ln(p_k!) - \sum_{i=1}^{C^{\text{home}}} \ln(q_i!).
\end{equation}

On the collection side, each of the $C^{\mathrm{pickup}}$ unique customers generates exactly one collection trip, regardless of how many parcels they collect---just as the carrier makes one stop at a pickup point regardless of how many parcels it delivers. Since each trip is individual and unconsolidated---each customer travels independently to their assigned service point---the collection side is maximally fragmented:

\begin{equation*}
\label{eq:collection_entropy}
    G^{\mathrm{pickup}} = \ln(C^{\text{pickup}}!).
\end{equation*}

The total system entropy is:
\begin{equation*}
\label{eq:total_entropy}
    G^{\mathrm{total}} = G^{\mathrm{delivery}} + G^{\mathrm{pickup}} = \left[\ln(N!) - \sum_{k=1}^{K} \ln(p_k!) - \sum_{i=1}^{C^{\mathrm{home}}} \ln(q_i!)\right] + \ln(C^{\mathrm{pickup}}!).
\end{equation*}

We normalize by the home delivery baseline, in which every customer receives exactly one parcel delivered to their home address and no collection trips occur:

\begin{equation*}
\label{eq:normalized_entropy}
    G^{\text{norm}} = \frac{G^{\text{total}}}{\ln(N!)}.
\end{equation*}

\begin{proposition}[System Entropy Increases Under Spatial Consolidation]
\label{prop:activity_increase}
Let $N$ parcels each belong to a unique customer ($C^{\text{total}} = N$). Under pure home delivery ($C^{\mathrm{pickup}} = 0$, $q_i = 1$ for all $i$), the baseline entropy is $G^{\mathrm{total}} = \ln(N!)$. Under pure pickup consolidation ($C^{\mathrm{pickup}} = N$), with parcels distributed across $K > 1$ service points:

\begin{equation*}
    G^{\text{total}} = 2\ln(N!) - \sum_{k=1}^{K} \ln(p_k!) > \ln(N!).
\end{equation*}

Total system entropy strictly exceeds the home delivery baseline.
\end{proposition}

\begin{proof}
Under pure home delivery, $G^{\mathrm{delivery}} = \ln(N!)$ and $G^{\mathrm{pickup}} = \ln(0!) = 0$, giving $G^{\mathrm{total}} = \ln(N!)$. Under pure pickup consolidation, $G^{\mathrm{delivery}} = \ln(N!) - \sum_{k=1}^{K} \ln(p_k!)$ and $G^{\mathrm{pickup}} = \ln(N!)$, giving $G^{\mathrm{total}} = 2\ln(N!) - \sum_{k=1}^{K} \ln(p_k!)$. Since $\sum_{k=1}^{K} \ln(p_k!) < \ln(N!)$ for any non-trivial partitioning ($K > 1$), the result follows.
\end{proof}

\begin{insight}[The Structural Asymmetry of Spatial Consolidation]
\label{insight:asymmetry}
Proposition~\ref{prop:activity_increase} reveals why spatial consolidation increases total system entropy despite reducing the carrier's operational complexity. Under home delivery, only the carrier moves---$N$ delivery stops, zero customer trips. Under pickup consolidation, the carrier makes fewer stops, but $N$ customers each make an individual collection trip. The carrier benefits from consolidation (many parcels per stop), but customer collection trips are inherently unconsolidated (one trip per customer). When customers have multiple parcels ($C^{\mathrm{pickup}} < N_{\mathrm{pickup}}$), the collection-side entropy is reduced through natural consolidation, but total system entropy still exceeds the baseline whenever $C^{\mathrm{pickup}} > 0$ and $K > 1$. Theoretically speaking, when no (fewer) collection trips are induced, entropy could decrease (be conserved), see Section~\ref{ssec:interpretation_regimes} that discusses the implications of \textit{trip chaining}.
\end{insight}

\subsubsection{Temporal Consolidation and Entropy Reduction}
Unlike spatial consolidation, which increases total system entropy by activating customer movements, temporal consolidation acts as a mechanism for genuine complexity destruction. By synchronizing delivery times, the system collapses the possible ordering permutations of the delivery sequence without creating new movements (see Figure~\ref{fig:temporal_consolidation}).

\begin{proposition}[Temporal Entropy Reduction] 
\label{prop:temporal_sink} 
Let $N$ parcels be temporally consolidated into $C \leq N$ discrete delivery events. The structural entropy of the delivery system is reduced from $\ln(N!)$ to $\ln(C!)$: 

\begin{equation*} 
G^{\mathrm{batched}} = \ln(C!) \le \ln(N!) = G^{\mathrm{original}} \end{equation*} 
\end{proposition}

\begin{proof} The configuration space of a logistics system is determined by the number of possible permutations of its deliveries. In the unconsolidated state, each of the $N$ parcels constitutes an independent delivery, yielding $W=N!$.

Temporal consolidation imposes a synchronization constraint such that multiple parcels are handled as a single delivery event. For the routing engine, the $C$ consolidated events become the new primary deliveries. The number of ways to sequence these events is: 

\begin{equation*} 
W^{\mathrm{batched}} = C! 
\end{equation*} 

Since the natural logarithm is a monotonically increasing function and $C \leq N$, it follows that: 

\begin{equation*} 
\ln(C!) \le \ln(N!) 
\end{equation*} 
Unlike spatial consolidation, which increases system-wide entropy by activating customer movements, temporal consolidation fundamentally reduces the cardinality of the state space without creating new movements. Temporal consolidation can also reduce customer-side entropy when customers combine multiple collection events into fewer trips.
\end{proof}

\begin{corollary}[Customer Temporal Consolidation]
\label{cor:cust_temporal}
Suppose the $C$ delivery events from Proposition \ref{prop:temporal_sink} are distributed across $K$ pickup points with allocation $(c_1,\dots,c_K)$, serving $C^{\mathrm{pickup}}$ unique customers. If customers consolidate their retrievals into $U \leq C^{\mathrm{pickup}}$ collection trips, then the collection-side entropy decreases:
\[
G^{\mathrm{pickup}} = \ln(U!) \leq \ln(C^{\mathrm{pickup}}!).
\]
\end{corollary}

\begin{proof}
As $U \leq C^{\mathrm{pickup}}$, the inequality trivially holds.
\end{proof}

% \begin{proof}
% \textbf{Proof:} Since $U \leq C^{\text{pickup}}$ and $\ln(\cdot)$ is increasing, $\ln(U!) \leq \ln(C^{\text{pickup}}!)$. \hfill $\Box$
% \end{proof}

\begin{insight}[Creation vs.\ Reduction]
\label{insight:creation_reduction}
Propositions \ref{prop:activity_increase} and \ref{prop:temporal_sink} reveal a fundamental distinction between spatial and temporal consolidation. Spatial consolidation creates entropy and temporal consolidation decreases entropy. 
\end{insight}

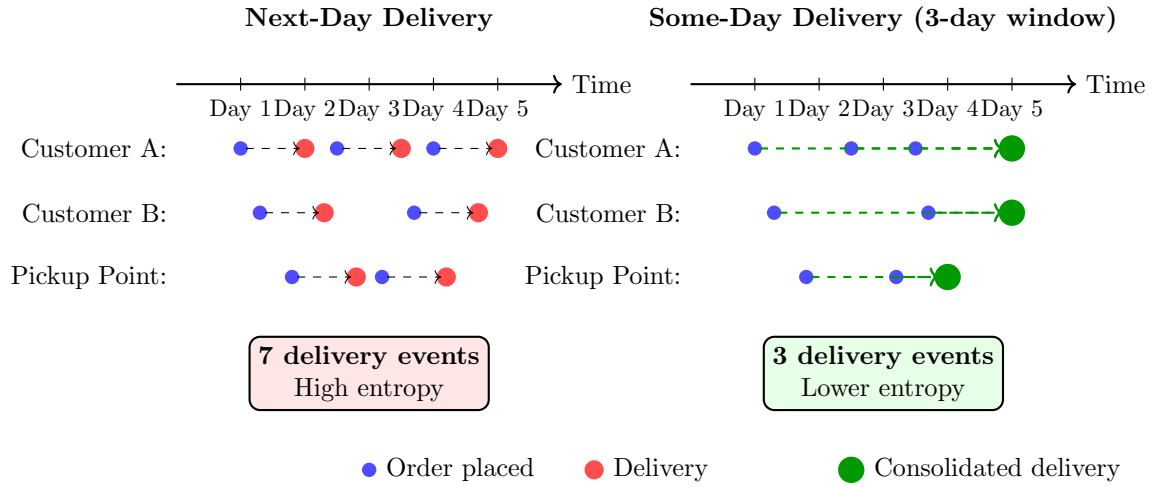
\begin{figure}[H]
\centering
\hspace*{-1.5cm}
\begin{tikzpicture}[
    order/.style={circle, fill=blue!70, minimum size=6pt, inner sep=0pt},
    delivery/.style={circle, fill=red!70, minimum size=8pt, inner sep=0pt},
    consolidated delivery/.style={circle, fill=green!60!black, minimum size=11pt, inner sep=0pt},
    scale=0.85,
    every node/.style={scale=0.9}
]

% Left panel: Next-day delivery
\node[font=\bfseries] at (3, 7) {Next-Day Delivery};

% Timeline
\draw[->, thick] (0, 6) -- (6, 6) node[right] {Time};
\foreach \x/\label in {1/Day 1, 2/Day 2, 3/Day 3, 4/Day 4, 5/Day 5} {
    \draw (\x, 5.9) -- (\x, 6.1);
    \node[below, font=\small] at (\x, 5.9) {\label};
}

% Customer A timeline
\node[left] at (0, 5) {Customer A:};
\node[order] at (1, 5) {}; 
\node[delivery] at (2, 5) {};
\draw[->, dashed] (1.08, 5) -- (1.92, 5);

\node[order] at (2.5, 5) {}; 
\node[delivery] at (3.5, 5) {};
\draw[->, dashed] (2.58, 5) -- (3.42, 5);

\node[order] at (4, 5) {}; 
\node[delivery] at (5, 5) {};
\draw[->, dashed] (4.08, 5) -- (4.92, 5);

% Customer B timeline
\node[left] at (0, 4) {Customer B:};
\node[order] at (1.3, 4) {}; 
\node[delivery] at (2.3, 4) {};
\draw[->, dashed] (1.38, 4) -- (2.22, 4);

\node[order] at (3.7, 4) {}; 
\node[delivery] at (4.7, 4) {};
\draw[->, dashed] (3.78, 4) -- (4.62, 4);

% Pickup point timeline
\node[left] at (0, 3) {Pickup Point:};
\node[order] at (1.8, 3) {}; 
\node[delivery] at (2.8, 3) {};
\draw[->, dashed] (1.88, 3) -- (2.72, 3);

\node[order] at (3.2, 3) {}; 
\node[delivery] at (4.2, 3) {};
\draw[->, dashed] (3.28, 3) -- (4.12, 3);

% Summary
\node[draw, thick, fill=red!10, align=center, rounded corners] at (3, 1.5) {
    \textbf{7 delivery events}\\
    High entropy
};

% Right panel: Some-day delivery
\node[font=\bfseries] at (11, 7) {Some-Day Delivery (3-day window)};

% Timeline
\draw[->, thick] (8, 6) -- (14, 6) node[right] {Time};
\foreach \x/\label in {9/Day 1, 10/Day 2, 11/Day 3, 12/Day 4, 13/Day 5} {
    \draw (\x, 5.9) -- (\x, 6.1);
    \node[below, font=\small] at (\x, 5.9) {\label};
}

% Customer A timeline - consolidated
\node[left] at (8, 5) {Customer A:};
\node[order] at (9, 5) {}; 
\node[order] at (10.5, 5) {}; 
\node[order] at (11.5, 5) {}; 
\node[consolidated delivery] at (13, 5) {};
\draw[->, dashed, thick, green!60!black] (9.08, 5) -- (12.86, 5);
\draw[->, dashed, thick, green!60!black] (10.58, 5) -- (12.86, 5);
\draw[->, dashed, thick, green!60!black] (12.08, 5) -- (12.86, 5);

% Customer B timeline - consolidated
\node[left] at (8, 4) {Customer B:};
\node[order] at (9.3, 4) {}; 
\node[order] at (11.7, 4) {}; 
\node[consolidated delivery] at (13, 4) {};
\draw[->, dashed, thick, green!60!black] (9.38, 4) -- (12.86, 4);
\draw[->, dashed, thick, green!60!black] (11.78, 4) -- (12.86, 4);

% Pickup point timeline - consolidated
\node[left] at (8, 3) {Pickup Point:};
\node[order] at (9.8, 3) {}; 
\node[order] at (11.2, 3) {}; 
\node[consolidated delivery] at (12, 3) {};
\draw[->, dashed, thick, green!60!black] (9.88, 3) -- (11.86, 3);
\draw[->, dashed, thick, green!60!black] (11.28, 3) -- (11.86, 3);

% Summary
\node[draw, thick, fill=green!10, align=center, rounded corners] at (11, 1.5) {
    \textbf{3 delivery events}\\
    Lower entropy
};

% Legend
\node[order, label=right:Order placed] at (3, 0) {};
\node[delivery, label=right:Delivery] at (6.5, 0) {};
\node[consolidated delivery, label=right:Consolidated delivery] at (10.5, 0) {};

\end{tikzpicture}
\caption{Temporal consolidation through \textit{some}-day delivery. System entropy is reduced as fewer transport movements are needed for the same deliveries.}
\label{fig:temporal_consolidation}
\end{figure}

\subsection{Generalized Entropy under Delivery Heterogeneity}
\label{subsec:generalized_entropy}

We now extend the entropy definition to accommodate heterogeneous delivery mechanisms frequently observed in last-mile logistics, including failed deliveries, pickup point restrictions, direct-to-pickup flows, and customer trip chaining. Rather than analyzing each operational variant separately, we present a unified approach within which these mechanisms arise as special cases.

\subsubsection{Generalized Entropy}

Consider a logistics system with $N$ distinguishable parcels partitioned into $R$ classes (e.g., home delivery, pickup point delivery). For each class $r$, let $N_r$ denote the number of parcels (where $\sum_{r=1}^R N_r = N$) and let $S_r \subseteq \{1,2,\dots,K\}$ denote the set of eligible pickup points. The fraction of parcels in class $r$ designated for home delivery is denoted by $\eta_r \in [0,1]$. Home delivery attempts fail independently with probability $\alpha \in [0,1]$, and at most $\Lambda \geq 1$ delivery attempts are made before rerouting to a pickup point.

Define $N_r^\mathrm{home} = \eta_r N_r$ as the number of parcels initially attempted via home delivery, and $N_r^\mathrm{pickup} = (1-\eta_r)N_r$ as the number sent directly to pickup points. Successful home deliveries equal $h_r = (1-\alpha^{\Lambda})N_r^\mathrm{home}$, 
while failed deliveries requiring pickup are $f_r = \alpha^{\Lambda}N_r^\mathrm{home}$. Let $p_{r,k}$ denote the number of parcels from class $r$ assigned to pickup point $k \in S_r$, satisfying $\sum_{k \in S_r} p_{r,k} = f_r + N_r^\mathrm{pickup}$, and let $p_k = \sum_{r: k \in S_r} p_{r,k}$ be the total parcels at pickup point $k$. Let $\tau_i$ denote the number of delivery attempts made for parcel $i$, and define the expected number of attempts $\mathbb{E}[\tau] = \frac{1-\alpha^{\Lambda}}{1-\alpha}$. Let $C^{\mathrm{pickup}}$ denote the total number of unique customers assigned to pickup points across all classes.

The total system entropy is defined as:

\begin{equation*}
G^\mathrm{total} =
\underbrace{\sum_{r=1}^{R} \left[
\ln\!\big((N_r^\mathrm{home} \, \mathbb{E}[\tau])!\big)
- N_r^\mathrm{home} \, \mathbb{E}[\ln(\tau!)]
- \ln(h_r!)
- \sum_{k \in S_r} \ln(p_{r,k}!)
\right]}_{G^{\mathrm{delivery}}}
+
\underbrace{\ln(C^{\mathrm{pickup}}!)}_{G^{\mathrm{pickup}}}
\end{equation*}

The delivery-side entropy $G^{\text{delivery}}$ captures the combinatorial complexity of sequencing delivery attempts and distributing parcels across locations. The collection-side entropy $G^{\mathrm{pickup}} = \ln(C^{\mathrm{pickup}}!)$ captures the system-wide movement generated by customers collecting their parcels. Each unique customer makes exactly one collection trip, regardless of how many parcels they retrieve. Because the realized delivery attempt counts $\{\tau_i\}$ are stochastic and not known ex ante, we adopt an expected entropy representation in which realization-dependent quantities are replaced by their expectations. This yields an ex ante structural entropy suitable for planning-stage analysis.

\subsubsection{Special Cases and Limiting Regimes}\label{ssec:special_cases}

The generalized entropy formulation subsumes several operational scenarios frequently studied in last-mile logistics. By appropriate parameter selection, previously defined delivery structures arise as special cases of the unified framework.

\medskip

\noindent\textbf{Baseline (no failures, full home delivery).}
When delivery failures are absent ($\alpha = 0$) and all parcels are designated for home delivery ($\eta_r = 1$ for all classes), each parcel requires exactly one delivery attempt and no pickup allocations occur. With $C^{\text{pickup}} = 0$, the collection-side entropy vanishes. Consequently,
\[
G^{\mathrm{total}} = \ln(N!),
\]
the pure home delivery baseline.

\medskip

\noindent\textbf{Single-attempt failure model.}
Setting $\Lambda = 1$ and $\eta_r = 1$ yields a regime in which every parcel receives at most one delivery attempt before rerouting to pickup. Failed parcels generate $C^{\mathrm{pickup}}$ unique customers at pickup points. Under symmetric allocation across $K$ pickup points, the total entropy becomes
\[
G^{\mathrm{total}} = \ln\!\big([N(1+\alpha)]!\big) 
- 2K \ln\!\big([\alpha N/K]!\big)
+ \ln(C^{\mathrm{pickup}}!).
\]

\medskip

\noindent\textbf{Multiple-attempt failure model.}
Allowing multiple delivery attempts ($\Lambda > 1$) introduces additional heterogeneity through the attempt counts $\{\tau_i\}$. Assuming symmetric pickup eligibility ($|S_r| = K$), the total entropy becomes
\[
G^{\mathrm{total}} = \ln\!\big([N\mathbb{E}[\tau]]!\big)
- \mathbb{E}[\ln(\tau!)]
- 2K \ln\!\big([N\alpha^{\Lambda}/K]!\big)
+ \ln(C^{\mathrm{pickup}}!),
\]
highlighting how delivery retries increase structural complexity through expanded attempt permutations.

\medskip

\noindent\textbf{Direct-to-pickup deliveries.}
When all parcels are sent directly to pickup points ($\eta_r = 0$), delivery entropy reduces to the multinomial allocation of distinguishable parcels across eligible pickup locations. With all $C^{\text{total}}$ customers collecting from pickup points, the total entropy simplifies to
\[
G^{\mathrm{total}} 
= \sum_r \left[\ln(N_r!) - \sum_{k\in S_r}\ln(p_{r,k}!)\right]
+ \ln(C^{\mathrm{total}}!).
\]

\medskip

\noindent\textbf{Single-attempt mixed delivery.}
When $\Lambda = 1$ but both home delivery and pickup routing are allowed ($\eta_r \in (0,1]$), entropy reflects the combinatorial split between successful home deliveries and parcels routed to pickup locations:
\[
G^{\mathrm{total}} 
= \sum_r \left[\ln(N_r!) - \ln(h_r!) - \sum_{k\in S_r}\ln(p_{r,k}!)\right]
+ \ln(C^{\mathrm{pickup}}!).
\]

\medskip

\noindent\textbf{Fully heterogeneous regime.}
For general $\eta_r$ and $\Lambda$, the full formulation as defined in Section~\ref{ssec:special_cases} applies. In this regime, structural entropy captures the combined effects of delivery retries, outcome heterogeneity, routing restrictions, and customer-side collection activity within a single unified combinatorial framework.

\subsection{Interpretation of different regimes on the level of entropy}
\label{ssec:interpretation_regimes}

The unified formulation reveals how common last-mile practices influence structural entropy through their effect on the underlying configuration space and the system-wide movement pattern.

\begin{itemize}
\item \emph{Failed deliveries} introduce additional branching in delivery outcomes, increasing structural dispersion by expanding the space of feasible delivery sequences. They also increase collection-side entropy by routing additional customers to pickup points, raising $C^{\text{pickup}}$.

\item \emph{Pickup eligibility restrictions} reduce delivery-side entropy by shrinking the feasible assignment sets $S_r$. Limiting the number of admissible pickup locations reduces the number of combinatorially distinct parcel-to-location configurations.

\item \emph{Direct-to-pickup flows} modify entropy through deterministic channel allocation. While direct routing to pickup points bypasses possible failed deliveries, it activates collection trips for all affected customers, increasing $G^{\mathrm{pickup}}$.

\item \emph{Trip chaining} reduces collection-side entropy by lowering the effective number of independent collection trips. When parcel retrieval is coupled to pre-existing trips (e.g., a trip to the grocery store), the customer does not generate an additional movement, reducing $C^{\text{pickup}}$ in the collection entropy. Theoretically speaking, trip chaining of \textit{all} collection trips, would maximally reduce system entropy (i.e., $C^\mathrm{pickup}=0$). Partial trip chaining provides a spectrum between entropy reduction, conservation, and increase (i.e., for some $C^\mathrm{chained} < C^\mathrm{pickup}$ entropy is conserved).
\end{itemize}

Although these mechanisms differ operationally, they act through a common structure: each modifies either the delivery-side configuration space, the collection-side movement count, or both. We now establish that total structural entropy grows linearly with system size under fixed structural proportions.

\begin{proposition}[Linear entropy scaling]
\label{prop:linear_scaling}
Suppose that as $N \to \infty$:

\begin{itemize}
\item class proportions $a_r = N_r/N$ remain fixed,
\item delivery outcome shares $h_r/N_r$, $f_r/N_r$, and $N_r^{(P)}/N_r$ remain fixed,
\item allocation shares $p_{r,k}/N$ remain fixed,
\item the ratio $C^{\text{pickup}}/N$ remains fixed.
\end{itemize}

Then total entropy satisfies
\[
G^{\mathrm{total}} = N \cdot c(\Theta) + \mathcal{O}(\ln N),
\]
where $\Theta$ denotes the collection of structural parameters,
and $c(\Theta)$ is a finite constant.
\end{proposition}

\begin{proof}
Applying Stirling's approximation $\ln(n!) = n\ln n - n + \mathcal{O}(\ln n)$
to each factorial term in the entropy expression---including $\ln(C^{\text{pickup}}!)$---yields contributions of order $n\ln n$. Under fixed proportion assumptions, the leading $N\ln N$ terms combine to produce a linear scaling in $N$, while logarithmic remainders accumulate to $\mathcal{O}(\ln N)$.
\end{proof}

Proposition~\ref{prop:linear_scaling} shows that structural dispersion grows proportionally with system scale under stable operational shares. Different delivery practices therefore modify the coefficient $c(\Theta)$ without changing the asymptotic order of growth.

\section{Case Study: Empirical Validation using Amazon Routing Data} \label{sec:case_study}

This section evaluates the proposed entropy framework using the Amazon Last-Mile Routing Research Challenge dataset \citep{amazon}. Comprising 6,112 high-resolution routes across five heterogeneous U.S. metropolitan areas---Seattle, Los Angeles, Chicago, Boston, and Austin---this dataset provides a representative cross-section of modern, high-density last-mile operations. By leveraging actual sequence data and package characteristics, we move beyond theoretical constructs to quantify the complexity inherent in current delivery practices.

To assess the utility of structural entropy as a diagnostic instrument, our experimental design employs a three-stage analytical pipeline:

\begin{enumerate} 
\item \textbf{Baseline Characterization (Section~\ref{ssec:current_state_entropy}):} We establish an empirical benchmark of the as-is state, quantifying the prevailing levels of fragmentation and demand balance across different urban topologies. 
\item \textbf{Counterfactual Scenario Analysis (Section~\ref{ssec:spatial_consolidation_analysis}):} Using Seattle as a focal geography, we simulate spatial consolidation via pickup-point activation. We apply a sigmoid-based behavioral model to test the sensitivity of system entropy to varying levels of customer adoption. \item \textbf{Operational Cross-Validation (Section~\ref{ssec:correlations}):} We perform a correlation analysis between normalized entropy and traditional performance indicators. Specifically, we investigate the existence of a stable scaling relationship linking entropy to total route distance. 
\end{enumerate}

\subsection{The Current State of Entropy}\label{ssec:current_state_entropy}

To establish an empirical baseline, we characterize the $6,112$ routes across three primary dimensions: volume ($N$), stop density ($K$), and the resulting entropy profiles ($G, H$). We normalize structural entropy as $G^{\mathrm{norm}} = G / \ln(N!)$ and Shannon entropy as $H^{\mathrm{norm}} = H / \ln(K)$. Since the baseline consists entirely of home deliveries ($C^{\mathrm{pickup}} = 0$), the collection-side entropy is zero and total system entropy equals delivery-side entropy. Table~\ref{tab:entropy_statistics} summarizes these metrics, revealing a network dominated by systematic fragmentation and high demand uniformity.

\begin{table}[ht]
\centering
\caption{Entropy and fragmentation statistics across 6,112 Amazon last-mile routes. Since the baseline consists of home deliveries only, $G^{\mathrm{pickup}} = 0$ and total entropy equals delivery entropy. The current system is characterized by high fragmentation (high $G$) and high demand uniformity (high $H$).}
\label{tab:entropy_statistics}
\begin{tabular}{lrrrr}
\toprule
\textbf{Metric} & \textbf{Mean} & \textbf{Std Dev} & \textbf{Min} & \textbf{Max} \\
\midrule
$N$ & 238.40 & 30.98 & 150.00 & 304.00 \\
$K$ & 147.99 & 31.03 & 33.00 & 238.00 \\
$N/K$ & 1.68 & 0.42 & 1.00 & 113.00 \\
\midrule
$G$ & 978.77 & 161.45 & 436.05 & 1356.82 \\
$G^{\text{norm}}$ & 0.91 & 0.04 & 0.62 & 0.98 \\
\midrule
$H$ & 4.78 & 0.30 & 3.00 & 5.39 \\
$H^{\text{norm}}$ & 0.96 & 0.02 & 0.71 & 0.99 \\
\bottomrule
\end{tabular}
\end{table}

With an average consolidation ratio of $1.68$ parcels per stop, $65\%$ of all stops serve a single parcel. This granularity is quantified by a mean $G^{\text{norm}}$ of $0.91$. The narrow standard deviation ($\sigma = 0.04$) suggests that this high-disorder state is a consistent operational characteristic across all five metropolitan areas. Furthermore, the high mean Shannon entropy ($H^{\text{norm}} = 0.96$) confirms that demand is distributed near-uniformly across stops. Only $1.6\%$ of stops exceed five parcels, typically representing commercial hubs or pickup points. As illustrated in Figure~\ref{fig:structural_entropy}, the tight clustering in the high-$G$, high-$H$ quadrant (from Table \ref{tab:entropy_quadrants}) represents a fragmentation trap: a balanced but operationally complex state where the carrier faces near-maximum combinatorial sequencing challenges. However, the observed range of $G^{\mathrm{norm}}$ (down to $0.62$) proves that significant consolidation is achievable, providing a quantitative basis for identifying best-in-class routing practices within the existing infrastructure.

\begin{figure}
    \centering
    \includegraphics[width=1.0\linewidth]{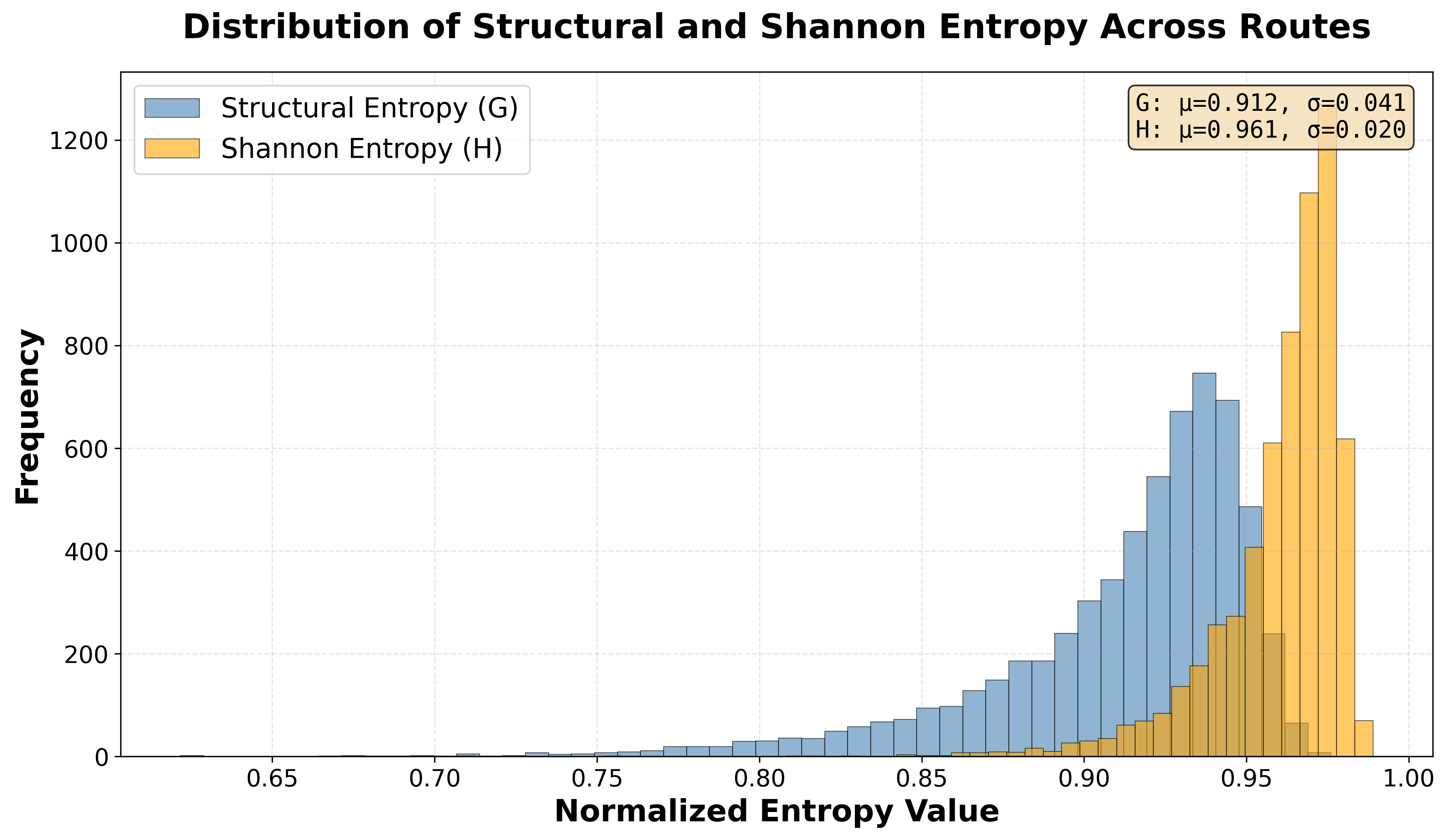}
    \caption{Distribution of Structural (Boltzmann) and Shannon entropy across routes. Table~\ref{tab:entropy_quadrants} classifies this entropy profile as high dispersion (High $G$) and balanced yet fragmented demand (High $H$). However, the lower values of $G$ do reveal significantly more consolidated routes.}
    \label{fig:structural_entropy}
\end{figure}

\subsection{Counterfactual Analysis: Spatial Consolidation and Behavioral Sensitivity}
\label{ssec:spatial_consolidation_analysis}
Since the Amazon dataset reflects current operations, the extent of existing consolidation remains latent. While high-volume stops (e.g., 6 parcels) may hint at temporal batching or commercial density, they often represent single-customer bulk orders. To explore the potential for structural change, we perform a counterfactual enrichment using the city of Seattle as a focal geography. By integrating pickup point locations via the Google Places API, we first identify that 2.6\% of current stops coincide with known consolidation hubs, accounting for 4.75\% of total parcel volume. In the remainder of this section, we assume these consolidation hubs as pickup points.

\subsubsection{Customer Activation Modeling}To quantify the latent potential for spatial consolidation, we model customer adoption as a stochastic process driven by proximity. Analysis of the Seattle network reveals a mean distance of $874$m ($\sigma = 740$m) to the nearest pickup point. We characterize customer willingness to adopt pickup delivery using a reversed sigmoid activation function, $P_a(d, \beta, t)$, representing the probability of a customer at distance $d$ choosing consolidation:

\begin{equation*}
P_a(d, \beta, t) = \frac{1}{1 + e^{\beta(d - t)}}
\end{equation*}

where $t$ is the activation threshold (the distance at which adoption probability is $50\%$) and $\beta$ is the steepness parameter, reflecting sensitivity to distance. This behavioral-operational link allows us to separate the \textit{activated population} ($A$)---those within a behavioral sweet spot---from the \textit{actual adopters} ($\lambda \cdot A$), where $\lambda$ represents the final acceptance ratio (see Appendix~\ref{numerical_results} for details).

\subsubsection{Impact on Delivery-Side Entropy}

The entropy metric $G^{\mathrm{norm}}$ reported in this section reflects delivery-side entropy as defined in Equation~\eqref{eq:delivery_entropy}, normalized by the home delivery baseline $\ln(N!)$. We focus on the carrier's perspective because the collection-side entropy is fully determined by the number of adopting customers and carries no additional structural information beyond the adoption count itself. Hence, delivery-side entropy therefore isolates the consolidation effect that drives operational outcomes, while the collection-side entropy can be easily reconstructed. 

Table \ref{tab:summary_activation} summarizes the outcomes across different behavioral regimes where  $G^{\mathrm{norm}}$ denotes the entropy of the counterfactual operations (full sensitivity results are provided in Appendix \ref{numerical_results}).

\begin{table}[ht]
\centering
\caption{Summary of System Entropy Sensitivity to Spatial Consolidation (Seattle). $G^{\text{norm}}$ includes only the delivery-side, normalized by $\ln(N!)$. Baseline $G^{\mathrm{norm}} = 0.912$. $t$ in kilometers.}
\label{tab:summary_activation}
\begin{tabular}{lccc}
\toprule
\textbf{Consolidation Regime} & \textbf{Parcel \% Shifted} & $\mathbf{G^{\mathrm{norm}}}$ & \textbf{Reduction} \\
\midrule
Baseline (Current State)      & 4.75\%           & 0.912          & --           \\ \addlinespace
Conservative ($t=0.25$)   & 4.8 -- 38.0\%    & 0.906 -- 0.762 & 0.7 -- 16.4\%  \\ \addlinespace
Moderate ($t=0.5$)        & 9.7 -- 43.3\%    & 0.894 -- 0.729 & 2.0 -- 20.1\% \\ \addlinespace
Aggressive ($t=1.0$)      & 13.8 -- 70.4\%   & 0.880 -- 0.547 & 3.5 -- 40.0\% \\
\bottomrule
\end{tabular}
\end{table}

The results demonstrate that delivery-side entropy reduction is non-linear and requires a critical mass of adoption. Moving less than $5\%$ of parcels to pickup points yields a modest ${\sim}0.6$--$0.8\%$ reduction in $G^{\text{norm}}$, whereas $10\%$ consolidation triggers approximately a $2\%$ improvement. Operationally, a target $G^{\text{norm}}$ of $0.75$--$0.80$ (an $11$--$18\%$ reduction) represents a challenging, yet feasible frontier for carriers, requiring a spatial consolidation of $28$--$42\%$. Reaching more ambitious targets (e.g., $G^{\text{norm}} < 0.70$) would necessitate consolidating over $50\%$ of parcels, implying a structural shift in pickup point density and additional customer incentives to adopt spatial consolidation. However, as established in Proposition~\ref{prop:activity_increase}, these delivery-side gains do not come free at the system level: each adopting customer generates a collection trip that did not previously exist, and total system entropy increases with the degree of spatial consolidation. Temporal consolidation remains the only mechanism that reduces system entropy simultaneously, as asserted in Corollary~\ref{cor:cust_temporal}.

\subsection{Operational Cross-Validation Against Other Metrics}\label{ssec:correlations}
To establish the decision-relevance of structural entropy ($G^{\mathrm{norm}}$), we examine its relationship with traditional logistics performance indicators. As expected from the mathematical definition, $G^{\mathrm{norm}}$ exhibits a near-perfect negative correlation with the standard deviation of parcels per stop (correlation coefficient $\rho = -0.951$, p-value $p^{\mathrm{val}} < 0.001$), confirming that entropy accurately captures the transition from lumpy consolidation to fragmented uniformity. Additionally, under a uniform parcel distribution, the relationship between the stop count $K$ and entropy $G$ is given by $G = \ln(N!) - K\ln(\frac{N}{K}!) \approx N \ln(K)$ (via Stirling). It is easily verifiable that two routes with identical $N$ and $K$ can have any $G$ value in the range $0 \leq G\leq N\ln(K)$, depending on the allocation vector $(p_1, ..., p_k)$. Hence, $G$ encodes the distribution information that $K$ discards.

More critically, entropy provides a structural explanation for route inefficiency that volume alone cannot provide. We find a moderately strong correlation between entropy and route compactness (the inverse of mean inter-stop distance; $\rho = -0.632, p^{\mathrm{val}}  < 0.001$) and total route distance ($\rho = 0.697, p^{\mathrm{val}}  < 0.001$). Notably, the number of parcels per route ($N$) shows a negligible correlation with total distance ($\rho = 0.092, p^{\mathrm{val}}  < 0.001$), suggesting that the spatial arrangement and consolidation state of the demand---quantified by entropy---is the primary driver of route length.

As illustrated in Figure~\ref{fig:entropy_vs_distance}, the relationship between normalized structural entropy and total route distance $d$ follows a non-linear scaling relationship:

\begin{equation*}
d = \kappa \cdot \frac{G^{\mathrm{norm}}}{1 - G^{\mathrm{norm}}}
\end{equation*}

where $\kappa$ is a scaling constant. Across all five metropolitan areas, $\kappa$ exhibits remarkable stability, with a mean $\bar{k} = 20.8$ km and a narrow range between $18.8$ and $24.6$ km. Further regression analysis indicates that depot proximity and local road topology do not significantly explain the variance in $\kappa$ ($R^2 < 0.03$).

The consistency of $\kappa$ across diverse urban environments suggests it functions as a network-inherent scaling factor. This finding validates $G^{\mathrm{norm}}$ as a predictive KPI: as entropy approaches the theoretical maximum ($G^{\mathrm{norm}} \to 1$), the marginal distance cost of adding fragmented demand accelerates asymptotically.

\begin{figure}[ht!]\centering\includegraphics[width=0.8\linewidth]{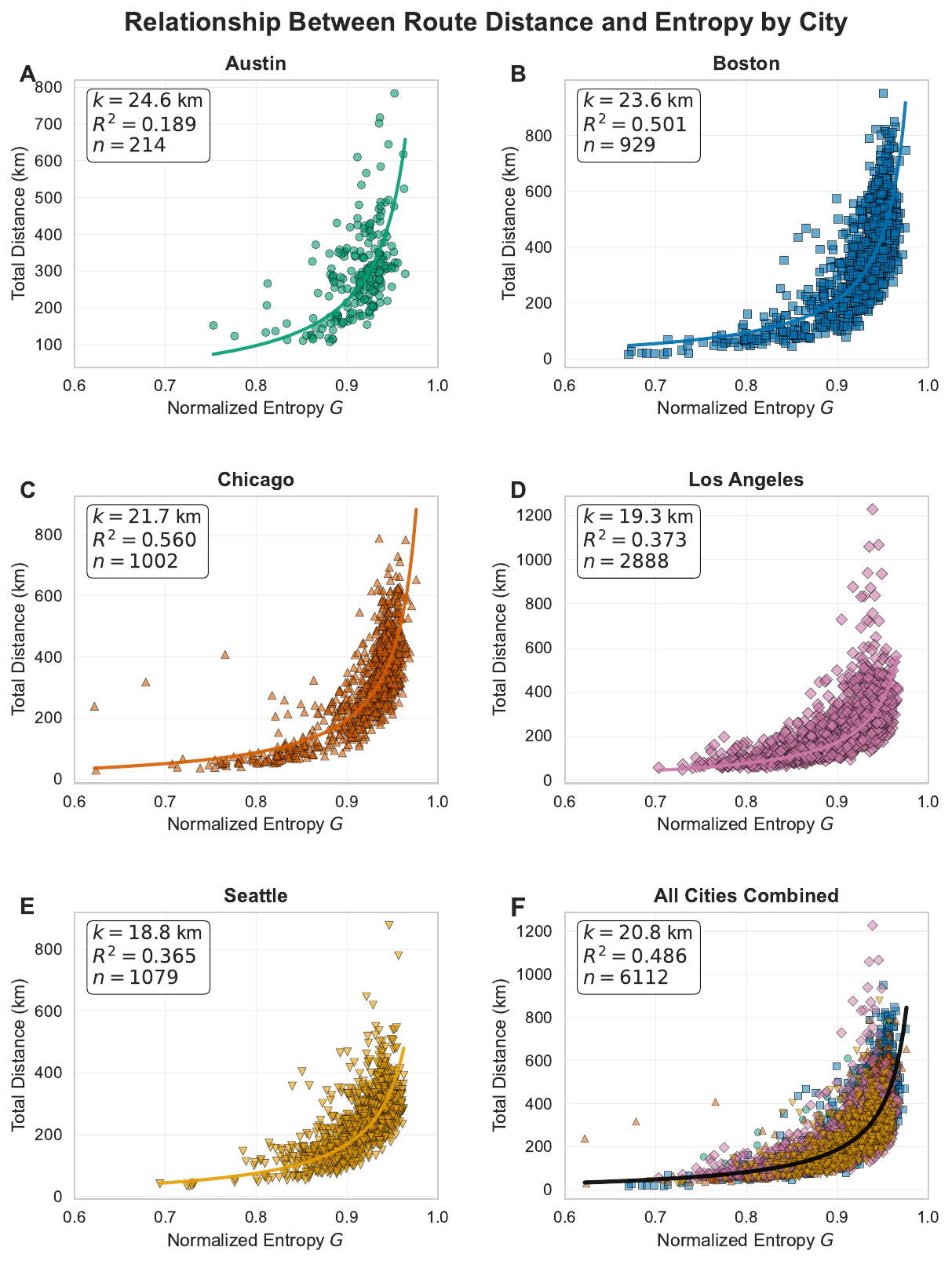}
\caption{Empirical validation of the entropy-distance scaling relationship across 6,112 Amazo delivery routes. Panels A--E present city-specific regressions for Seattle, Los Angeles, Chicago, Boston, and Austin, demonstrating the consistency of the scaling factor $\kappa$. Panel F shows the aggregate dataset fit ($\kappa = 20.8$ km). The asymptotic behavior as $G^{\mathrm{norm}} \to 1$ quantifies the extreme distance penalty associated with total system fragmentation.}\label{fig:entropy_vs_distance}
\end{figure}

\section{Managerial Implications and Discussion}
\label{sec:discussion}

The entropy measure introduced in this study captures a system property orthogonal to traditional cost and service metrics: the structural complexity inherent in fragmented demand. Whereas routing optimization focuses on how to serve a given set of delivery points efficiently, entropy explains why certain demand configurations create unavoidable inefficiency. This distinction shifts managerial attention from purely operational optimization toward structural intervention.

\subsection{Implications for Carriers}

For carriers, entropy functions as a forward-looking diagnostic rather than an ex post performance indicator. The scaling relation established in Section~\ref{ssec:correlations} ($d \propto \kappa \cdot \frac{G}{1-G}$) links structural disorder directly to expected route mileage, allowing managers to quantify the distance penalty associated with demand fragmentation before vehicles are dispatched. The empirical stability of $\kappa \approx 20.8$ km across five cities suggests that entropy provides a robust, transferable indicator of structural routing difficulty.

Our baseline analysis (Section~\ref{ssec:current_state_entropy}) exemplifies that contemporary last-mile logistics operations reside in a high-fragmentation regime ($G^{\mathrm{norm}} = 0.912$), where incremental routing optimization yields diminishing returns. At this level of disorder, efficiency gains require structural change rather than algorithmic refinement. Two complementary strategies emerge.

First, entropy can be reduced temporally by consolidating demand over time, for example through scheduled delivery days (e.g., Amazon Day on a chosen day of the week). Temporal aggregation lowers disorder at the source by reshaping the arrival process of orders, reducing entropy for all system actors simultaneously.

Second, entropy can be managed spatially through the use of lockers, hubs, or neighborhood consolidation points. As demonstrated in the Seattle simulation, spatial consolidation reduces the carrier's delivery-side entropy by redirecting parcels to shared service points. However, the system-wide entropy measure reveals that this private gain comes at a system-level cost: each consolidated parcel generates a customer collection trip, and total system entropy may increase (Proposition~\ref{prop:activity_increase}). Carriers should therefore recognize spatial consolidation as an entropy transfer mechanism---beneficial to the carrier, but not costless to the system---and evaluate it in conjunction with temporal strategies that achieve genuine system-wide entropy reduction.

Embedding entropy into decision-support systems enables three concrete applications. Entropy provides a reliability signal: routes approaching maximal disorder ($G^{\mathrm{norm}} \to 1$) are inherently brittle and sensitive to disruption, enabling proactive targeting of consolidation incentives. Because normalized entropy is scale-invariant, it also enables strategic benchmarking across heterogeneous urban environments, helping managers distinguish between poor execution and structurally difficult markets. Finally, the generalized entropy measure decomposes dispersion into distinct mechanisms---failed deliveries, pickup eligibility design, and direct-to-pickup routing (Section~\ref{ssec:special_cases})---allowing targeted interventions such as improving first-attempt success rates, redesigning eligibility regions, or adjusting pickup density rather than relying solely on route optimization.

\subsection{Policy Implications: Managing Urban Logistics Complexity}

From a municipal perspective, entropy reframes urban logistics challenges by targeting the structural drivers of congestion rather than downstream symptoms such as vehicle counts. High-entropy delivery systems require disproportionately more road-time per parcel due to the non-linear scaling identified in Section~\ref{ssec:correlations}, amplifying environmental and social externalities.

Entropy therefore offers a technology-neutral regulatory metric. Instead of prescriptive restrictions (e.g., vehicle bans), municipalities could introduce entropy-based access thresholds in dense urban areas. According to Table~\ref{tab:entropy_quadrants}, carriers operating with high structural entropy signal highly fragmented operations and correspondingly high externality exposure. Performance-based pricing schemes could link congestion charges to structural disorder, rewarding carriers that reduce entropy through consolidation, collaboration, or infrastructure investment. 

The system-wide entropy perspective is particularly relevant for policy design. Policies promoting spatial consolidation should account for the induced customer mobility: redirecting deliveries to pickup points reduces carrier entropy but generates collection trips that contribute to pedestrian traffic and local congestion around service points. The total system entropy may increase even as the carrier's private entropy decreases. Effective policy should therefore prioritize temporal consolidation---which reduces entropy for all actors, see Corollary~\ref{cor:cust_temporal}---and evaluate spatial consolidation through the lens of total system entropy rather than carrier-side entropy alone.

The generalized entropy decomposition further enables regulators to distinguish between structural and operational sources of dispersion. Policies targeting first-attempt delivery success reduce entropy by eliminating redundant delivery cycles without restricting economic activity. Similarly, coordinated pickup zoning and infrastructure placement can lower spatial entropy by enabling trip chaining---which reduces collection-side entropy by lowering $C^{\mathrm{pickup}}$---and reducing isolated delivery nodes.

\subsection{Consumer Implications: Redefining the Convenience Frontier}

From the consumer perspective, entropy highlights a fundamental trade-off between individual convenience and collective urban livability. Unconstrained home delivery maintains a high-entropy equilibrium on the delivery side, effectively externalizing logistics complexity into congestion, emissions, and neighborhood disturbance. However, from the consumer's perspective, it imposes zero collection burden.

Our counterfactual analysis (Section~\ref{ssec:spatial_consolidation_analysis}) shows that transitions toward lower carrier entropy depend strongly on distance sensitivity. The sigmoid activation model identifies a behavioral threshold beyond which spatial consolidation becomes operationally viable. Communicating this threshold reframes consumers as active participants in system-level efficiency: proximity to consolidation infrastructure and customer willingness to accept out-of-home delivery options directly determines the system's ability to reduce carrier travel distance.

The system-wide entropy measure adds an important nuance. When a consumer opts for pickup delivery, they reduce the carrier's entropy contribution but generate a collection trip that increases total system entropy. The net system benefit depends on the relative externalities of these movements: a delivery van navigating a congested urban core imposes different costs than a customer walking 500 meters to a locker. This asymmetry suggests that spatial consolidation is most beneficial in contexts where carrier trips carry disproportionate externalities relative to customer collection trips---for instance, in dense city centers where van access is restricted or emissions regulations apply.

Entropy also enables new incentive mechanisms. By quantifying the marginal structural disorder associated with a specific delivery request, carriers can design pricing schemes that internalize externalities for consumers. High-entropy requests---such as isolated home deliveries or narrow delivery windows---may carry higher prices reflecting their systemic impact, whereas lower-entropy options, including green delivery slots or neighborhood pickup hubs, can be incentivized. Making structural complexity visible allows consumers to balance personal convenience with measurable community benefits, redefining the convenience--externality frontier.

\section{Conclusion}
\label{sec:conclusion}

This paper introduces structural entropy (derived from Boltzmann mechanics) as a quantitative measure of dispersion in last-mile logistics systems. Rather than evaluating routes solely through realized costs or distance, the framework characterizes the combinatorial configuration space induced by parcel allocations and delivery mechanisms. By linking structural dispersion to routing distance through a scaling relationship, the analysis connects demand fragmentation to measurable operational consequences.

Our results yield three primary contributions to the logistics literature:

\begin{enumerate}
\item \textbf{A Predictive Scaling Relationship:} We show that route distance is governed not only by volume but by the dispersion structure of deliveries. Across five heterogeneous U.S. cities, route length exhibits a stable non-linear relationship with structural entropy, with an estimated scaling constant ($\kappa \approx 20.8$ km). This empirical regularity suggests that dispersion captures a persistent structural feature of last-mile networks.

\item \textbf{Evidence of High-Dispersion Operations:} An analysis of 6,112 Amazon routes indicates that current last-mile systems operate near highly fragmented configurations ($G^{\mathrm{norm}} \approx 0.91$). In such regimes, incremental routing improvements yield limited gains, while structural changes in delivery configuration (e.g., strategically placed parcel lockers, Amazon Day for temporally consolidated deliveries) may offer larger performance effects.

\item \textbf{A System-Wide View of Consolidation:} We establish that spatial consolidation under idealized symmetric conditions conserves structural entropy---a theoretical result that reveals the mechanism by which service points transfer complexity across the firm boundary. Both symmetry conditions (unconstrained customer choice and indistinguishable parcels) are violated in practice. Our system-wide entropy measure, which accounts for all movements by both carriers and customers, shows that spatial consolidation increases total system entropy by activating customer collection trips that did not previously exist. However, when customers combine parcel collection with pre-existing trips---a behavior known as \textit{trip chaining}---the collection-side entropy is reduced, since chained trips do not constitute additional independent movements. Depending on the degree of trip chaining, system entropy may still increase (no chaining), may be conserved (partial chaining), or genuinely decrease (full chaining), making trip chaining a critical determinant of whether spatial consolidation achieves system-wide efficiency gains. Temporal consolidation, by contrast, always reduces entropy for all actors by decreasing the number of delivery events. This distinction provides a principled basis for evaluating logistics interventions at the system level.
\end{enumerate}

Several directions for future research follow naturally. First, incorporating dispersion measures into vehicle routing formulations may yield solutions that are robust to fragmentation rather than solely distance-optimal ex post. Second, empirical validation in other geographical settings (e.g., rural areas, cities in Europe and Asia) would test the stability of the estimated scaling relationship across urban morphologies. Third, the distributional implications of consolidation policies warrant further investigation, particularly if incentives to reduce dispersion shift retrieval effort toward specific customer groups. Fourth, a weighted entropy that differentiates between movement types and transport modes by their social and environmental cost would enable a complete welfare analysis of spatial consolidation strategies.

By formalizing demand fragmentation as a measurable structural property, entropy provides a tractable bridge between combinatorial system configuration and operational performance. In doing so, it offers a basis for evaluating last-mile policies not only by their local routing effects, but by their impact on the overall dispersion structure of the network.

\bibliography{bib}

@article{brochado2024performance,
  title={Performance evaluation and explainability of last-mile delivery},
  author={Brochado, {\^A}ngela F and Rocha, Eug{\'e}nio M and Addo, Emmanuel and Silva, Samuel},
  journal={Procedia Computer Science},
  volume={232},
  pages={2478--2487},
  year={2024},
  publisher={Elsevier}
}

@article{vanheeswijk2019delivery,
  title={The delivery dispatching problem with time windows for urban consolidation centers},
  author={Van Heeswijk, WJA and Mes,  MRK and Schutten, MJ},
  journal={Transportation Science},
  volume={53},
  number={1},
  pages={203--221},
  year={2019},
  publisher={INFORMS}
}

@book{Clausius1879Mechanical,
  author    = {Clausius, Rudolf},
  title     = {The Mechanical Theory of Heat},
  publisher = {Macmillan \& Co.},
  year      = {1879},
  address   = {London}
}

@article{gudmundsson2013entropy,
  title = {Entropy and Order in Urban Street Networks},
  author = {Gudmundsson, Agust and Mohajeri, Nahid},
  journal = {Scientific Reports},
  volume = {3},
  pages = {3324},
  year = {2013},
  publisher = {Nature Publishing Group},
  doi = {10.1038/srep03324},
  url = {https://doi.org/10.1038/srep03324}
}

@article{halldorsson2020last,
  title={Last-mile logistics fulfilment: A framework for energy efficiency},
  author={Halldorsson, Arni and Wehner, Jessica},
  journal={Research in Transportation Business \& Management},
  volume={37},
  pages={100481},
  year={2020},
  publisher={Elsevier}
}

@article{abbas2006maximum,
  title = {Maximum Entropy Utility},
  author = {Abbas, Ali E.},
  journal = {Operations Research},
  volume = {54},
  number = {2},
  pages = {277--290},
  year = {2006},
  publisher = {INFORMS},
  doi = {10.1287/opre.1050.0253},
  url = {https://doi.org/10.1287/opre.1050.0253}
}

@Article{boltzmannentropy,
AUTHOR = {Sharp, Kim and Matschinsky, Franz},
TITLE = {Translation of Ludwig Boltzmann’s Paper “On the Relationship between the Second Fundamental Theorem of the Mechanical Theory of Heat and Probability Calculations Regarding the Conditions for Thermal Equilibrium” Sitzungberichte der Kaiserlichen Akademie der Wissenschaften. Mathematisch-Naturwissen Classe. Abt. II, LXXVI 1877, pp 373-435 (Wien. Ber. 1877, 76:373-435). Reprinted in Wiss. Abhandlungen, Vol. II, reprint 42, p. 164-223, Barth, Leipzig, 1909},
JOURNAL = {Entropy},
VOLUME = {17},
YEAR = {2015},
NUMBER = {4},
PAGES = {1971--2009},
URL = {https://www.mdpi.com/1099-4300/17/4/1971},
ISSN = {1099-4300},
 publisher =    {MDPI},
DOI = {10.3390/e17041971}
}

@InProceedings{ahmed-entropy-regulation,
  title = 	 {Understanding the Impact of Entropy on Policy Optimization},
  author =       {Ahmed, Zafarali and Le Roux, Nicolas and Norouzi, Mohammad and Schuurmans, Dale},
  booktitle = 	 {Proceedings of the 36th International Conference on Machine Learning},
  pages = 	 {151--160},
  year = 	 {2019},
  editor = 	 {Chaudhuri, Kamalika and Salakhutdinov, Ruslan},
  volume = 	 {97},
  series = 	 {Proceedings of Machine Learning Research},
  month = 	 {09--15 Jun},
  publisher =    {PMLR},
  url = 	 {https://proceedings.mlr.press/v97/ahmed19a.html}
}

@article{lim2025cutting,
  title={Cutting Last-Mile Delivery Costs},
  author={Lim, Stanley Frederick WT},
  journal={MIT Sloan Management Review},
  volume={66},
  number={2},
  pages={15--17},
  year={2025},
  publisher={Massachusetts Institute of Technology, Cambridge, MA}
}

@article{vanheeswijk2020evaluating,
  title={Evaluating urban logistics schemes using agent-based simulation},
  author={Van Heeswijk, WJA and Mes, MaRK and Schutten, JMJ and Zijm, WHM},
  journal={Transportation Science},
  volume={54},
  number={3},
  pages={651--675},
  year={2020},
  publisher={INFORMS}
}

@article{savelsbergh201650th,
  title={50th anniversary invited article—city logistics: Challenges and opportunities},
  author={Savelsbergh, Martin and Van Woensel, Tom},
  journal={Transportation Science},
  volume={50},
  number={2},
  pages={579--590},
  year={2016},
  publisher={INFORMS}
}

@article{shannon1948,
  author    = {Shannon, Claude E.},
  title     = {A Mathematical Theory of Communication},
  journal   = {Bell System Technical Journal},
  volume    = {27},
  number    = {3},
  pages     = {379--423},
  year      = {1948},
  doi       = {10.1002/j.1538-7305.1948.tb01338.x}
}

@article{amazon,
  author    = {Merch{\'a}n, Daniel and Arora, Jatin and Pach{\'o}n, Julian and Konduri, Karthik and Winkenbach, Matthias and Parks, Steven and Noszek, Joseph},
  title     = {2021 Amazon Last Mile Routing Research Challenge: Data Set},
  journal   = {Transportation Science},
  volume    = {58},
  number    = {1},
  pages     = {8--11},
  year      = {2024},
  doi       = {10.1287/trsc.2022.1173},
 publisher={INFORMS}
}

@article{snyder2004,
author = {Snyder, Lawrence},
year = {2004},
month = {09},
pages = {},
title = {Facility Location Under Uncertainty: A Review},
volume = {38},
journal = {IIE Transactions},
doi = {10.1080/07408170500216480}
}

@article{matl2017,
  author    = {Matl, P. and Hartl, R. F. and Vidal, T.},
  title     = {Workload Equity in Vehicle Routing Problems: A Survey and Analysis},
  journal   = {Transportation Science},
  volume    = {52},
  number    = {2},
  pages     = {239--260},
  year      = {2017},
  publisher={INFORMS}
}

@article{janjevic2017investigating,
  title={Investigating the theoretical cost-relationships of urban consolidation centres for their users},
  author={Janjevic, Milena and Ndiaye, Alassane},
  journal={Transportation Research Part A: Policy and Practice},
  volume={102},
  pages={98--118},
  year={2017},
  publisher={Elsevier}
}

@Article{su14020911,
AUTHOR = {Zennaro, Ilenia and Finco, Serena and Calzavara, Martina and Persona, Alessandro},
TITLE = {Implementing E-Commerce from Logistic Perspective: Literature Review and Methodological Framework},
JOURNAL = {Sustainability},
VOLUME = {14},
YEAR = {2022},
NUMBER = {2},
ARTICLE-NUMBER = {911},
URL = {https://www.mdpi.com/2071-1050/14/2/911},
ISSN = {2071-1050},
DOI = {10.3390/su14020911}
}

@article{Mangiaracina2019,
    author = {Mangiaracina, Riccardo and Perego, Alessandro and Seghezzi, Arianna and Tumino, Angela},
    title = {Innovative solutions to increase last-mile delivery efficiency in B2C e-commerce: a literature review},
    journal = {International Journal of Physical Distribution \& Logistics Management},
    volume = {49},
    number = {9},
    pages = {901-920},
    year = {2019},
    month = {10},
    issn = {0960-0035},
    doi = {10.1108/IJPDLM-02-2019-0048},
    url = {https://doi.org/10.1108/IJPDLM-02-2019-0048},
    eprint = {https://www.emerald.com/ijpdlm/article-pdf/49/9/901/1109218/ijpdlm-02-2019-0048.pdf},
}

@article{wang2018innovation,
  title={An innovation diffusion perspective of e-consumers’ initial adoption of self-collection service via automated parcel station},
  author={Wang, Xueqin and Yuen, Kum Fai and Wong, Yiik Diew and Teo, Chee Chong},
  journal={The International Journal of Logistics Management},
  volume={29},
  number={1},
  pages={237--260},
  year={2018},
  publisher={Emerald Publishing Limited}
}

@article{VegaMejia2019,
  author    = {Vega-Mej{\'\i}a, Carlos A. and Montoya-Torres, Jairo R. and Islam, S. M. N.},
  title     = {Consideration of triple bottom line objectives for sustainability in the optimization of vehicle routing and loading operations: a systematic literature review},
  journal   = {Annals of Operations Research},
  volume    = {273},
  pages     = {311--375},
  year      = {2019},
  doi       = {10.1007/s10479-017-2723-9}
}

@article{anderson,
author = {Anderson, Michael and Müller, Anna},
year = {2025},
month = {10},
pages = {114-129},
title = {Causal Discovery in Multi-Echelon Supply Networks: Leveraging Foundation Models for Demand Propagation Analysis},
volume = {2},
journal = {Frontiers in Applied Physics and Mathematics},
doi = {10.71465/fapm416},
 publisher =    {Frontiers Media SA}
}

@article{amaral2020,
title = {An exploratory evaluation of urban street networks for last mile distribution},
journal = {Cities},
volume = {107},
pages = {102916},
year = {2020},
issn = {0264-2751},
publisher = {Elsevier},
doi = {https://doi.org/10.1016/j.cities.2020.102916},
url = {https://www.sciencedirect.com/science/article/pii/S0264275120312646},
author = {Julia Coutinho Amaral and Claudio B. Cunha}
}

\newpage 

\section*{Appendices}

\section*{Stirling Approximation for Structural Entropy}
\label{appendix_stirling}

This appendix derives the asymptotic expansion underlying the linear scaling
result in Proposition~\ref{prop:linear_scaling} and clarifies the order of
magnitude of higher-order corrections.

Recall the structural entropy definition
\begin{equation*}
    G = \ln \left(\frac{N!}{\prod_{k=1}^K p_k!}\right)
    = \ln N! - \sum_{k=1}^K \ln p_k!,
\end{equation*}
where $N \in \mathbb{Z}_{>0}$ is the total number of parcels,
$p_k \in \mathbb{Z}_{\ge 0}$ denotes the number of parcels assigned to
destination $k$, and $\sum_{k=1}^K p_k = N$.
Define empirical shares $\phi_k := p_k/N$ so that $\sum_{k=1}^K \phi_k = 1$.

\subsection*{Leading-order expansion}

We use Stirling's expansion
\[
\ln k! = k \ln k - k + \tfrac{1}{2}\ln(2\pi k)
+ \frac{1}{12k}
- \frac{1}{360k^3}
+ \mathcal{O}(k^{-5}).
\]

Retaining only the leading two terms yields
\[
\ln k! \approx k \ln k - k.
\]

Substituting into $G$ gives
\begin{align*}
G 
&\approx (N \ln N - N)
 - \sum_{k=1}^K (p_k \ln p_k - p_k).
\end{align*}

The linear $-k$ terms cancel exactly due to the conservation constraint
$\sum_{k=1}^K p_k = N$.
Hence
\begin{align*}
G 
&\approx N \ln N - \sum_{k=1}^K p_k \ln p_k \\
&= N \ln N - N \sum_{k=1}^K \phi_k \ln (N \phi_k) \\
&= - N \sum_{k=1}^K \phi_k \ln \phi_k.
\end{align*}

Thus,
\[
G = N H(\{\phi_k\}) + \mathcal{O}(\ln N),
\]
where
\[
H(\{\phi_k\}) = -\sum_{k=1}^K \phi_k \ln \phi_k
\]
is the Shannon entropy (in nats) of the empirical distribution.

Under fixed shares $\{\phi_k\}$, the dominant contribution to structural
entropy scales linearly in $N$.

\subsection*{Logarithmic correction}

Including the $\tfrac{1}{2}\ln(2\pi k)$ term gives
\[
G 
\approx
N H(\{\phi_k\})
+
\frac{1}{2}
\left[
\ln(2\pi N)
-
\sum_{k=1}^K \ln(2\pi p_k)
\right].
\]

This correction term is of order $\mathcal{O}(\ln N)$ when the number of occupied
destinations $K$ grows sublinearly in $N$ or when delivery shares remain
stable.
It captures finite-size effects and depends on the dispersion of the
allocation vector through $\sum_k \ln p_k$.

In the special case of a uniform distribution
$p_k = N/K$ for all $k$, we obtain
\[
G 
\approx
N \ln K
+
\frac{1}{2}
\left[
\ln(2\pi N)
-
K \ln\!\left(2\pi \frac{N}{K}\right)
\right].
\]

\subsection*{Higher-order correction}

Including the $1/(12k)$ term yields
\[
G
\approx
N H(\{\phi_k\})
+
\frac{1}{2}
\left[
\ln(2\pi N)
-
\sum_{k=1}^K \ln(2\pi p_k)
\right]
+
\left[
\frac{1}{12N}
-
\sum_{k=1}^K \frac{1}{12 p_k}
\right].
\]

The final bracket scales as $\mathcal{O}(K)$ and becomes relevant only when a large
number of destinations carry very small parcel counts.
Under fixed structural proportions (as assumed in
Proposition~\ref{prop:linear_scaling}),
this term remains asymptotically negligible relative to the $\mathcal{O}(N)$ leading
contribution.

\medskip

In summary, structural entropy admits the asymptotic expansion
\[
G = N H(\{\phi_k\}) + \mathcal{O}(\ln N),
\]
which justifies the linear scaling behavior established in the main text.
Higher-order terms provide finite-size corrections but do not alter the
order of growth under stable operational shares.

\subsection*{Approximation accuracy}
We quantify the accuracy of the truncations based on realistic instances from the Amazon dataset used in the main text. For $N\in\{100,200,300\}$ we consider average parcels per occupied destination $\bar{p} = N/K$ equal to the values:
$\bar{p} \in \{1.0,\;1.5,\;2.0,\;5.0,\;10.0\}.$
For each pair $(N,\bar{p})$ we round to the nearest integer. We then compute the exact solution to
\begin{equation*}
    G = \ln N! - \sum_{k=1}^K \ln p_k!
\end{equation*}

using the log-gamma function, and the approximations obtained by applying additional terms of Stirling's expansion. Let T1 denote the approximation obtained by the two leading terms, T2 the approximation when the logarithmic correction is added, and T3 the approximation when the higher-order correction is added. From Table \ref{shirling_numerical} we observe that the T1 approximation is a poor estimate in highly fragmented logistics systems as the relative error can exceed 20\%. For numerical estimation this approximation thus overestimates the structural entropy. However, T1 approximation can still be used to compare relative differences between delivery scenarios as the additional linear terms in the T2 and T3 approximations cancel out when forming a ratio. The T2-approximation reduces the relative error to a percentage-level in all combinations and the T3-approximation gives almost exact results. Even higher order terms (i.e., $-1/(360k^3)$) can thus be assumed to be negligible in real-world contexts.

\begin{table}[H] \label{shirling_numerical}
\centering
\caption{Exact entropy $G$ and percentage error of Stirling approximations for different $N$ and average parcels per destination $\bar{p}$.}
\begin{tabular}{cccccc}
\hline
$N$ & $\bar{p}$ & $G^{\text{exact}}$ & Error (T1) & Error (T2) & Error (T3) \\
\hline
100 & 1   & 363.74 & +26.6\%  & +2.2\%   & -0.06\% \\
100 & 1.5 & 344.67 & +21.8\%  & +1.1\%   & -0.01\% \\
100 & 2   & 329.08 & +18.9\%  & +0.6\%   & -0.005\% \\
100 & 5   & 267.99 & +4.6\%   & +0.1\%   & -0.002\% \\
100 & 10  & 212.7  & +8.3\%   & +0.04\%  & $\approx$0\% \\
\hline
200 & 1   & 863.23 & +22.8\%  & +1.9\%   & -0.05\% \\
200 & 1.5 & 825.37 & +18.6\%  & +0.9\%   & -0.01\% \\
200 & 2   & 793.92 & +16.0\%  & +0.5\%   & -0.004\% \\
200 & 5   & 671.73 & +9.8\%   & +0.1\%   & $\approx$0\% \\
200 & 10  & 516.14 & +6.8\%   & +0.03\%  & $\approx$0\% \\
\hline
300 & 1   & 1414.91 & +20.9\% & +1.7\%   & -0.05\% \\
300 & 1.5 & 1357.97 & +17.1\% & +0.8\%   & -0.01\% \\
300 & 2   & 1310.93  & +14.7\% & +0.5\%   & -0.004\% \\
300 & 5   & 1127.66  & +8.9\%  & +0.1\%   & $\approx$0\% \\
300 & 10  & 961.77  & +6.1\%  & +0.03\%  & $\approx$0\% \\
\hline
\end{tabular}
\end{table}

\section*{Numerical Results} \label{numerical_results}

\begin{table*}[bp]
\centering
\caption{Normalized System Entropy by Activation Threshold $t$ (in km), Steepness Parameter $\beta$, and Customer Acceptance Ratio $\lambda$. $A$ denotes the activated customers and $\lambda \cdot A$ the percentage of customers using a pickup point. $G^{\mathrm{norm}}$ includes only delivery-side entropy, normalized by $\ln(N!)$. Baseline $G^{\mathrm{norm}} = 0.912$.}
\label{tab:activation_results}
\begin{tabular}{lccccccc}
\toprule
\textbf{$t$} & \textbf{$\beta$} & \textbf{$\lambda$} & \textbf{A} & \textbf{$\lambda \cdot A$ } & \textbf{Parcels} & $\mathbf{G^{\text{norm}}}$ & \textbf{Reduction} \\
 & & & & & \textbf{Consol.} & & \textbf{from baseline} \\
\midrule
\multirow{12}{*}{0.25}
& \multirow{4}{*}{1.0}
& 25\%  & 36.69\% & 9.17\%  & 9.38\%  & 0.894600 & 1.91\% \\
& & 50\%  & 36.68\% & 18.34\% & 18.81\% & 0.859337 & 5.77\% \\
& & 75\%  & 36.89\% & 27.67\% & 28.29\% & 0.813832 & 10.76\% \\
& & 100\% & 36.84\% & 36.84\% & 38.00\% & 0.761702 & 16.48\% \\
\cmidrule{2-8}
& \multirow{4}{*}{5.0}
& 25\%  & 19.76\% & 4.94\%  & 5.49\%  & 0.904806 & 0.79\% \\
& & 50\%  & 19.63\% & 9.82\%  & 11.09\% & 0.889510 & 2.47\% \\
& & 75\%  & 19.60\% & 14.70\% & 16.82\% & 0.868113 & 4.81\% \\
& & 100\% & 19.78\% & 19.78\% & 22.47\% & 0.843444 & 7.52\% \\
\cmidrule{2-8}
& \multirow{4}{*}{10.0}
& 25\%  & 15.67\% & 3.92\%  & 4.81\%  & 0.906493 & 0.60\% \\
& & 50\%  & 15.67\% & 7.84\%  & 9.57\%  & 0.894227 & 1.95\% \\
& & 75\%  & 15.53\% & 11.65\% & 14.10\% & 0.878749 & 3.65\% \\
& & 100\% & 15.77\% & 15.77\% & 19.14\% & 0.858693 & 5.85\% \\
\midrule
\multirow{12}{*}{0.5}
& \multirow{4}{*}{1.0}
& 25\%  & 42.43\% & 10.61\% & 11.00\% & 0.889463 & 2.47\% \\
& & 50\%  & 42.23\% & 21.12\% & 21.83\% & 0.845866 & 7.25\% \\
& & 75\%  & 42.63\% & 31.97\% & 32.88\% & 0.790484 & 13.32\% \\
& & 100\% & 42.31\% & 42.31\% & 43.32\% & 0.729373 & 20.02\% \\
\cmidrule{2-8}
& \multirow{4}{*}{5.0}
& 25\%  & 35.25\% & 8.81\%  & 9.70\%  & 0.894010 & 1.97\% \\
& & 50\%  & 35.46\% & 17.73\% & 19.40\% & 0.857356 & 5.99\% \\
& & 75\%  & 35.31\% & 26.48\% & 28.91\% & 0.811189 & 11.05\% \\
& & 100\% & 35.25\% & 35.25\% & 38.84\% & 0.757062 & 16.99\% \\
\cmidrule{2-8}
& \multirow{4}{*}{10.0}
& 25\%  & 35.42\% & 8.86\%  & 9.89\%  & 0.892912 & 2.09\% \\
& & 50\%  & 35.44\% & 17.72\% & 19.74\% & 0.854202 & 6.34\% \\
& & 75\%  & 35.44\% & 26.58\% & 29.62\% & 0.805374 & 11.69\% \\
& & 100\% & 35.41\% & 35.41\% & 39.67\% & 0.749594 & 17.81\% \\
\midrule
\multirow{12}{*}{1.0}
& \multirow{4}{*}{1.0}
& 25\%  & 53.69\% & 13.42\% & 13.80\% & 0.879528 & 3.56\% \\
& & 50\%  & 53.83\% & 26.92\% & 27.65\% & 0.817398 & 10.37\% \\
& & 75\%  & 53.81\% & 40.36\% & 41.46\% & 0.741054 & 18.74\% \\
& & 100\% & 53.84\% & 53.84\% & 55.03\% & 0.655486 & 28.13\% \\
\cmidrule{2-8}
& \multirow{4}{*}{5.0}
& 25\%  & 65.09\% & 16.27\% & 17.01\% & 0.866259 & 5.02\% \\
& & 50\%  & 65.00\% & 32.50\% & 33.81\% & 0.783959 & 14.04\% \\
& & 75\%  & 65.06\% & 48.80\% & 51.03\% & 0.680638 & 25.37\% \\
& & 100\% & 65.10\% & 65.10\% & 67.87\% & 0.567107 & 37.82\% \\
\cmidrule{2-8}
& \multirow{4}{*}{10.0}
& 25\%  & 67.81\% & 16.95\% & 17.64\% & 0.863785 & 5.29\% \\
& & 50\%  & 67.88\% & 33.94\% & 35.44\% & 0.774511 & 15.08\% \\
& & 75\%  & 67.75\% & 50.81\% & 52.97\% & 0.666396 & 26.93\% \\
& & 100\% & 67.68\% & 67.68\% & 70.44\% & 0.546588 & 40.07\% \\
\bottomrule
\end{tabular}
\end{table*}

\end{document}